\documentclass{amsart}

\usepackage{amssymb}
\usepackage{graphicx}
\usepackage{amsmath}
\usepackage{amsthm}
\usepackage{pgf,tikz}
\usepackage{mathrsfs}
\usepackage{ulem}
\usetikzlibrary{arrows}
\usepackage[all]{xy}
\usepackage[unicode=true,
 bookmarks=false,
 breaklinks=false,pdfborder={0 0 1},backref=section,colorlinks=false]
 {hyperref}

\newtheorem{theorem}{Theorem}[section]

\newtheorem{definition}[theorem]{Definition}
\newtheorem{example}[theorem]{Example}
\newtheorem{lemma}[theorem]{Lemma}

\newtheorem{proposition}[theorem]{Proposition}
\newtheorem{remark}[theorem]{Remark}

\begin{document}

\title[On a topological duality for Stonean Hilbert algebras]
{On a topological duality for Stonean Hilbert algebras}

\address{Facultad de Ciencias y Artes, Universidad Cat\'olica de \'Avila, C/Canteros s/n, 05005 \'Avila, Espa\~{n}a.}
\email{ismaelm.calomino@ucavila.es}
\author{Ismael Calomino}

\address{Departamento de Matem\'aticas, Facultad de Econom\'ia y Administraci\'on, Universidad Nacional del Comahue, C/Buenos Aires 1400, 8300 Neuqu\'en, Argentina.}
\email{dmontang@gmail.com}
\author{Daniela Montangie}


\begin{abstract}
The aim of this paper is to present a topological duality for the algebraic category of Stonean Hilbert algebras and certain morphisms between them, called $\neg$-morphisms. Also, we apply this duality to characterize the Heyting algebra of $\alpha$-ideals of a Stonean Hilbert algebra. 
\end{abstract}

\maketitle

\section{Introduction} 

The class of Hilbert algebras is the algebraic counterpart of the implicative fragment of intuitionistic propositional calculus and the implicative subreducts of Heyting algebras. Recall that a Hilbert algebra is an algebra ${\bf{A}} = \langle A, \to, 1 \rangle$ of type $(2,0)$ which satisfies the following conditions:
\begin{enumerate}
\item $a \to (b \to a) = 1$,
\item $(a \to (b \to c)) \to ((a \to b) \to (a \to c)) = 1$,
\item $a \to b = b \to a = 1$ implies $a = b$.
\end{enumerate}

In every Hilbert algebra we have the partial order $\leq$ given by $a \leq b$ if and only if $a \to b = 1$, called the natural order. Relative to the natural order on ${\bf{A}}$, $1$ is the greatest element, and if $A$ has a first element $0$, ${\bf{A}}$ is called bounded. Hilbert algebras form a variety (\cite{Diego}) and were studied by several authors in different directions: Hilbert algebras with lattice operations (\cite{BG,CM1,Gaitan1}), Hilbert algebras with additional operators (\cite{CM2,CM0,CSM,GLCC,CM4,Figallo}), prelinear Hilbert algebras (\cite{CCSM,SMS}) and Stonean Hilbert algebras (\cite{Saeid1,Saeid2,Gaitan}). 

A duality theory for Hilbert algebras was developed in \cite{CCM} through ordered topological spaces and simplified in \cite{CM1} using sober topological spaces with a basis of open and compact sets satisfying an additional condition, where the order can be defined with the topological closure. In particular, in \cite{CM1}, the authors develop a duality for the class of Hilbert algebras where the associated order is a join-semilattice, named Hilbert algebras with supremum or $H^{\vee}$-algebras. These results were of great importance to begin studying different modal style operators from a topological point of view (for more details see \cite{CM0,CSM,CM4}). Recently in \cite{Gaitan}, a topological representation is presented for the class of Stonean Hilbert algebras, which turn out to be bounded Hilbert algebras with supremum satisfying Stone's identity 
\[
\neg x \vee \neg \neg x = 1,
\]
where $\neg x = x \to 0$. The main goal of this note is to present an alternative topological duality  to that given in \cite{Gaitan} following the approach developed in \cite{CM1,CM4} and studying certain morphisms (called $\neg$-morphisms) that we consider more suitable to characterize the variety of Stonean Hilbert algebras. 

The paper is organized as follows. In Section \ref{sec2} we give some basic results about bounded Hilbert algebras with supremum and recall the topological duality given in \cite{CM1,CM4}. In Section \ref{sec3} we study the class of Stonean Hilbert algebras and give some characterizations through regular elements. In Section \ref{sec4} we introduce $\neg$-morphisms, which turn out to be morphisms different from $\vee0$-semi-homomorphisms and $\vee0$-homomorphisms. In Section \ref{sec5} we present a topological duality for the algebraic category of Stonean Hilbert algebras with $\neg$-morphisms as arrow through Stonean $H_{0}^{\vee}$-spaces and certain spectral maps that fulfill an additional condition. Finally, in Section \ref{sec6}, we apply the duality of Section \ref{sec5} to obtain a topological representation of Heyting algebra of $\alpha$-ideals of a Stonean Hilbert algebra.

\section{Preliminaries} \label{sec2}

\subsection{Bounded Hilbert algebras with supremum}

Recall the notion of bounded Hilbert algebra where the associated order is a join-semilattice (\cite{CM1,C1}).

\begin{definition}
A Hilbert algebra with supremum, or $H^{\vee}$-algebra for short, is an algebra ${\bf{A}} = \langle A, \to, \vee, 1 \rangle$ of type $(2,2,0)$ where $\langle A, \to, 1 \rangle$ is a Hilbert algebra, $\langle A, \vee, 1 \rangle$ is a join-semilattice with greatest element, and $a \to b = 1$ if and only if $a \vee b = b$, for all $a,b \in A$.
\end{definition}

\begin{theorem} \cite{C1} \label{caracterization}
Let us consider an algebra ${\bf{A}} = \langle A, \to, \vee, 1 \rangle$ of type $(2,2,0).$ Then ${\bf{A}}$ is an $H^{\vee}$-algebra if and only if
\begin{enumerate}
\item $\langle A, \to, 1 \rangle$ is a Hilbert algebra,
\item $\langle A, \vee, 1 \rangle$ is a join-semilattice with greatest element,
\item $A$ satisfies the following equations
\begin{enumerate}
\item $a \to (a \vee b) = 1$,
\item $(a \to b) \to ((a \vee b) \to b) = 1$.
\end{enumerate}
\end{enumerate}
\end{theorem}

An $H^{\vee}$-algebra ${\bf{A}}$ is bounded, or $H_{0}^{\vee}$-algebra for short, if there exists an element $0 \in A$ such that $0 \to a = 1$, for all $a \in A$. We denote by $\mathrm{Hil}_{0}^{\vee}$ the variety of $H_{0}^{\vee}$-algebras.

Let ${\bf{A}} \in \mathrm{Hil}_{0}^{\vee}$. A non-empty subset $F$ of $A$ is called an implicative filter if $1 \in F$, and if $a, a \to b \in F$ then $b \in F$. Denote by $\mathrm{Fi}(A)$ the set of all implicative filters of ${\bf{A}}$. If $C \subseteq A$, then the implicative filter ${\mathrm{Fig}} (C)$ generated by $C$ can be
characterized as the set
\[
{\mathrm{Fig}} (C) = \{  a \in A \colon \exists a_{1}, \ldots, a_{n} \in C \text{ such that } (a_{1}, \ldots, a_{n} ;a)  =1 \},
\]
where 
\[
(a_{1}, \ldots, a_{n}; a)  = 
\left\{
\begin{array} {llllll}
a_{1} \to a & \text{if} & n=1, \\
a_{1} \to (a_{2}, \ldots, a_{n}; a) & \text{if} & n>1. \\
\end{array}
\right.
\]
A proper implicative filter $x$ is said to be irreducible if for any $F_{1}, F_{2} \in \mathrm{Fi}(A)$ such that $x = F_{1} \cap F_{2}$ implies $x = F_{1}$ or $x = F_{2}$. Denote by $\mathrm{X}(A)$ the set of all irreducible implicative filters of ${\bf{A}}$. Recall that a proper filter $x$ is irreducible if and only if for every $a,b \in A$ such that $a,b \notin x$, there exists $c \notin x$ such that $a,b \leq c$. We say that a implicative filter $M$ is maximal if for any implicative filter $F$, if $M \subseteq F$ then $M=F$ or $M=A$. We denote by $\mathrm{Max(A)}$ the set of all maximal implicative filters of ${\bf{A}}$. Note that ${\mathrm{Max}(A)} \subseteq \mathrm{X}(A)$. A proper implicative filter $P$ is said to be prime if $a \vee b \in P$ implies that either $a \in P$ or $b \in P$, for all $a,b \in A$. It is straightforward to see that a proper implicative filter is irreducible if and only if is prime. A non-empty downset $I$ of $A$ is said to be an ideal if $0 \in I$ and for all $a,b \in I$ we have $a \vee b \in I$. The set of all ideals of ${\bf{A}}$ is denoted by $\mathrm{Id}(A)$. We denote by $\mathrm{Idg}(C)$ the ideal generated by a subset $C$ of $A$. Then we have 
\[
\mathrm{Idg}(C) = \{ a \in A \colon \exists c_{1}, \ldots, c_{n} \in C \text{ such that } a \leq c_{1} \vee \ldots \vee c_{n} \}.
\]

The following result, proved in \cite{Busneag,MonteiroSymetric}, will be useful in this paper.

\begin{theorem} \label{separacion}
Let ${\bf{A}} \in \mathrm{Hil}_{0}^{\vee}$. Let $F \in \mathrm{Fi}(A)$ and $I \in \mathrm{Id}(A)$ be such that $F \cap I = \emptyset$. Then there exists $x \in \mathrm{X}(A)$ such that $F \subseteq x$ and $x \cap I = \emptyset$.
\end{theorem}

Let ${\bf{A}} \in \mathrm{Hil}_{0}^{\vee}$. If $a \in A$, we shall write $\neg a := a \to 0$. 

\begin{lemma} \cite{Diego} \label{lemma_pri}
Let ${\bf{A}} \in \mathrm{Hil}_{0}^{\vee}$. Then the following properties are satisfied for all $a, b \in A$:
\begin{enumerate}
\item $a \leq \neg \neg a $,
\item if $a \leq b$ then $\neg b\leq  \neg a $,
\item $\neg a= \neg \neg \neg a $,
\item $a \rightarrow b \leq \neg b \rightarrow \neg a $,
\item $a \rightarrow \neg b = b\rightarrow \neg a $,
\item $\neg \neg (a \rightarrow b) \leq \neg \neg a \rightarrow \neg \neg b$,
\item $a \to \neg a = \neg a$.
\end{enumerate}
\end{lemma}

\begin{lemma} \cite{Diego,C1} \label{lema1}
Let ${\bf{A}} \in \mathrm{Hil}_{0}^{\vee}$. Let $a \in A$ and $x \in \mathrm{X}(A)$. Then:
\begin{enumerate}
\item If $a \in x$, then $\neg a \notin x$.
\item $\neg a \notin x$ if and only if there exists $y \in \mathrm{X}(A)$ such that $x \subseteq y$ and $a \in y$.
\item $\neg a \notin x$ if and only if there exists $M \in {\mathrm{Max}(A)}$ such that $x \subseteq M$ and $a \in M$.
\end{enumerate}
\end{lemma}

An element $a \in A$ is called dense if $\neg a=0$. Denote by $\mathrm{De}(A)$ the set of all dense elements of $A$, i.e., $\mathrm{De}(A) = \{ a \in A \colon \neg a=0 \}$. 

\begin{lemma} \cite{C1} \label{lema2}
Let ${\bf{A}} \in \mathrm{Hil}_{0}^{\vee}$. Let $M \in \mathrm{Fi}(A)$. Then the following conditions are equivalent:
\begin{enumerate}
\item $M$ is maximal,
\item For every $a \in A$, $a \notin M$ implies that $\neg a \in M$,
\item For every $a \in A$, $a \notin M$ implies that $\neg \neg a \notin M$,
\item $M$ is irreducible and $\mathrm{De}(A) \subseteq M$.
\end{enumerate}
\end{lemma}

\begin{definition}
Let ${\bf{A}}, {\bf{B}} \in \mathrm{Hil}_{0}^{\vee}$. A mapping $h \colon A \to B$ is a $\vee0$-semi-homomorphism if the following conditions are satisfied for every $a,b \in A$:
\begin{enumerate}
\item $h(1) = 1$, 
\item $h(0) = 0$,
\item $h(a \to b) \leq h(a) \to h(b)$, 
\item $h(a \vee b) = h(a) \vee h(b)$.
\end{enumerate} 
A $\vee0$-homomorphism is a $\vee0$-semi-homomorphism $h$ such that $h(a) \to h(b) \leq h(a \to b)$, for all $a,b \in A$.
\end{definition}

We consider the following two algebraic categories: 
\[
\begin{array}{llllllll}
{\mathcal{SH}}_{0}^{\vee} & = & \text{$H_{0}^{\vee}$-algebras} & + & \text{$\vee0$-semi-homomorphisms}, \\
{\mathcal{HH}}_{0}^{\vee} & = & \text{$H_{0}^{\vee}$-algebras} & + & \text{$\vee0$-homomorphisms}. \\
\end{array}
\]

\subsection{$H_{0}^{\vee}$-spaces}

In \cite{CCM,CM1} a full duality for the class of Hilbert algebras and the class of Hilbert algebras with supremum was developed. Recently in \cite{CM4}, the authors simplifies the duality given in \cite{CM1} where the morphisms are two special order-preserving maps. In this subsection we recall the duality for Hilbert algebras with supremum.

Let $\langle X, \mathcal{T}_{\mathcal{K}} \rangle$ be a topological space with a base $\mathcal{K}$ and let $D(X) = \{U \colon U^{c} \in \mathcal{K}\}$. 

\begin{definition} \cite{CM1}
An $H_{0}^{\vee}$-space is a topological space $\langle X, \mathcal{T}_{\mathcal{K}} \rangle$ such that:
\begin{enumerate}
\item[(H1)] $\mathcal{K}$ is a base of open and compact subsets for the topology $\mathcal{T}_{\mathcal{K}}$,
\item[(H2)] $U \cap V \in\mathcal{K}$ and $U \Rightarrow V := (U \cap V^{c}]^c  \in D(X)$, for all $U,V \in \mathcal{K}$,
\item[(H3)] $X \in \mathcal{K}$,
\item[(H4)] $\langle X, \mathcal{T}_{\mathcal{K}} \rangle$ is sober.
\end{enumerate}
\end{definition}

Let $\langle X, \leq \rangle$ be a poset. We denote by $\mathrm{Up}(X)$ the family of all upsets of $X$. We known that $\mathrm{H}(X) = \langle \mathrm{Up}(X), \Rightarrow, \cup ,\emptyset, X \rangle$ is a bounded Hilbert algebra with supremum, called the complex Hilbert algebra of $\langle X, \leq \rangle$.
If ${\bf{A}} \in \mathrm{Hil}_{0}^{\vee}$, then ${\bf{A}}$ is isomorphic to the subalgebra $D(\mathrm{X}(A)) := \{ \varphi(a) \colon a \in A \}$ of $\mathrm{H}(\mathrm{X}(A))$ via the map $\varphi \colon A \to \mathrm{Up}(\mathrm{X}(A))$ defined by $\varphi(a):= \{ x \in \mathrm{X}(A) \colon a \in x \}$. Taking into account \cite{CM1}, we conclude that $\langle \mathrm{X}(A), \mathcal{T}_{\mathcal{K}_{A}} \rangle$ is an $H_{0}^{\vee}$-space, which we call the dual space of ${\bf{A}}$, where the family $\mathcal{K}_{A} := \{  \varphi(a)^{c} \colon a \in A \}$ forms a basis of compact subsets.

\begin{remark} \label{obs varphi(neg a)}
Let ${\bf{A}} \in \mathrm{Hil}_{0}^{\vee}$ and let $\langle \mathrm{X}(A),\mathcal{T}_{\mathcal{K}_{A}} \rangle$ be the dual space of ${\bf{A}}$. For $a \in A$, note that 
\[
\begin{array}{llllllll}
\varphi (a \to 0) &=& \varphi(a) \Rightarrow \varphi(0) &=& \varphi(a) \Rightarrow \emptyset \\
                         &=& (\varphi(a) \cap \emptyset^c]^c &=& (\varphi(a)]^c,
\end{array}
\]
i.e., $\varphi( \neg a) = (\varphi(a)]^c$.
\end{remark}

Furthermore, if $\langle X, \mathcal{T}_{\mathcal{K}} \rangle $ is an $H_{0}^{\vee}$-space, then $\langle D(X), \cup, \Rightarrow, \emptyset, X \rangle \in \mathrm{Hil}_{0}^{\vee}$ and $\langle X, \mathcal{T}_{\mathcal{K}} \rangle$ is homeomorphic to the $H_{0}^{\vee}$-space $\langle \mathrm{X}(D(X)), \mathcal{T}_{\mathcal{K}_{D(X)}} \rangle$ via the map $\epsilon \colon X \to \mathrm{X}(D(X))$ given by $\epsilon (x):= \{  U \in D(X) \colon x \in U \}$.

Given two $H_{0}^{\vee}$-spaces $\langle X_{1}, \mathcal{T}_{\mathcal{K}_{1}} \rangle$ and $\langle X_{2}, \mathcal{T}_{\mathcal{K}_{2}} \rangle$ we can consider two special order-preserving maps $f \colon X_{1} \to X_{2}$:
\begin{itemize}
\item We say that $f$ is a spectral map if $f^{-1}(U) \in \mathcal{K}_{1}$, for all $U \in \mathcal{K}_{2}$.
\item We say that $f$ is a p-morphism if is a spectral map such that $[f(x)) = f([x))$, for all $x \in X$.
\end{itemize}
Then we have two categories: 
\[
\begin{array}{llllllll}
{\mathcal{SMS}}_{0}^{\vee} & = & \text{$H_{0}^{\vee}$-spaces} & + & \text{spectral maps}, \\
{\mathcal{PMS}}_{0}^{\vee} & = & \text{$H_{0}^{\vee}$-spaces} & + & \text{p-morphisms}. \\
\end{array}
\]
Then by the results given in \cite{CM4} the category ${\mathcal{SH}}_{0}^{\vee}$ is dually equivalent to the category ${\mathcal{SMS}}_{0}^{\vee}$, and the category ${\mathcal{HH}}_{0}^{\vee}$ is dually equivalent to the category ${\mathcal{PMS}}_{0}^{\vee}$.

\section{Stonean Hilbert algebras} \label{sec3}

\begin{definition} \cite{C1,Saeid1}
Let ${\bf{A}} \in \mathrm{Hil}_{0}^{\vee}$. We say that ${\bf{A}}$ is a Stonean Hilbert algebra if ${\bf{A}}$ satisfies the identity $\neg a \vee \neg \neg a = 1$.
\end{definition}

We denote by $\mathrm{SHil}_{0}^{\vee}$ the class of Stonean Hilbert algebras.

\begin{example} \label{exam1}
Let $\langle A, \vee, 0, 1 \rangle$ be a bounded join-semilattice. Consider on $A$ the implication $\to$ given by the order 
\[
a \to b  = 
\left\{
\begin{array} {llllll}
1 & \text{if} & a \leq b, \\
b & \text{if} & a \nleq b. \\
\end{array}
\right.
\]
It is straightforward to see that the structure $\langle A, \to, \vee, 0, 1 \rangle$ is a Stonean Hilbert algebra.
\end{example}

\begin{example} \label{ejemplolocal}
A Hilbert algebra is said to be local if it has exactly one maximal implicative filter. So, a bounded Hilbert algebra ${\bf{A}}$ is local if and only if $\mathrm{De}(A) = A - \{0\}$. It follows that every local bounded Hilbert algebra with supremum is a Stonean Hilbert algebra (see \cite{Saeid1,Saeid2}).  
\end{example}

\begin{example}
A Hilbert algebra is said to be prelinear if satisfies the identity $((a \to b) \to c) \to (((b \to a) \to c) \to c) = 1$. In particular, if ${\bf{A}} \in \mathrm{Hil}_{0}^{\vee}$, then ${\bf{A}}$ is prelinear if and only if $(a \to b) \vee (b \to a) = 1$, for all $a,b \in A$ (see \cite{Diego,CCSM}). Then, by Lemma \ref{lemma_pri}, it follows that 
\[
1 = (\neg \neg a \to \neg a) \vee (\neg a \to \neg \neg a) = \neg a \vee \neg \neg a,
\] 
for all $a \in A$. Therefore, every bounded prelinear Hilbert algebra with supremum is a Stonean Hilbert algebra. It is easy to see that there exist Stonean Hilbert algebras which are not prelinear ones (see Remark \ref{rem_negmorf}).
\end{example}

The following result gives different characterizations of Stonean Hilbert algebras.

\begin{lemma} \cite{C1} \label{lemaSergio}
Let ${\bf{A}} \in \mathrm{Hil}_{0}^{\vee}$. Then the following conditions are equivalent:
\begin{enumerate}
\item ${\bf{A}} \in \mathrm{SHil}_{0}^{\vee}$, 
\item Each irreducible implicative filter is contained in a unique maximal implicative filter, 
\item For any increasing subsets $U,V \subseteq \mathrm{X}(A)$, we have $(U] \cap (V] = (U \cap V]$,
\item $a \to \neg b = \neg a \vee \neg b$, for all $a,b \in A$.
\end{enumerate}
\end{lemma}

\begin{remark} 
If ${\bf{A}} \in \mathrm{SHil}_{0}^{\vee}$, then $[a) \cap [\neg a) \subseteq \mathrm{De}(A)$, for all $a \in A$. Indeed, let $b\in A$ such that $b \in [a) \cap [\neg a)$, then $a\leq b$ and $\neg a \leq b$. Thus, $\neg \neg a \leq \neg \neg b$ and $\neg \neg \neg a = \neg a \leq \neg \neg b$. So, $\neg a \vee \neg \neg a  \leq \neg \neg b$ and as $\mathbf{A}$ is a Stonean Hilbert algebra, $\neg \neg b = 1$ and consequently $\neg b = 0$. Therefore, $b \in \mathrm{De}(A)$.
\end{remark}

In \cite{BG}, it was proved that if ${\bf{A}}$ is a Hilbert algebra and $a,b \in A$ are such that $a \vee b$ exists, then, for every $c \in A$, the meet $(a \to c) \wedge (b \to c)$ exists and $(a \to c) \wedge (b \to c) = (a \vee b) \to c$. Thus, if ${\bf{A}} \in \mathrm{Hil}_{0}^{\vee}$ we have that there exists $(a \to c) \wedge (b \to c)$ and 
\[
(a \to c) \wedge (b \to c) = (a \vee b) \to c,
\]
for all $a,b,c \in A$ (see \cite{BG}, Corollary 4). In particular, if $c=0$ we have $\neg(a \vee b) = \neg a \wedge \neg b$ and therefore there exists $\neg a \wedge \neg b$, for all $a,b \in A$.

Let ${\bf{A}} \in \mathrm{Hil}_{0}^{\vee}$. Consider the equation 
\begin{equation} \label{SP} 
\neg a \to (\neg b \to (\neg a \wedge \neg b)) = 1.
\end{equation}
From the preceding observations, we know that $\neg a \wedge \neg b$ always exists. Hence, equation (\ref{SP}) is well-defined in ${\bf{A}}$. Our next goal is to prove that (\ref{SP}) is satisfied in every Stonean Hilbert algebra. We proceed to show that this condition holds by means of its regular elements. Recall that the set of regular elements of ${\bf{A}}$ is 
\[
R(A) = \{ \neg a \colon a \in A \} = \{a \in A \colon \neg \neg a = a\}.
\]
Then, the structure $R({\bf{A}}) = \langle R(A),  \veebar, \barwedge, \neg, 0, 1 \rangle$ is a Boolean algebra (see \cite{GT}), where the lattice operations $\veebar$ and $\barwedge$ are defined by 
\begin{eqnarray*}
a \barwedge b &:=& \neg (a \to \neg b), \\
a \veebar b &:=& \neg \neg (\neg a \to b).
\end{eqnarray*}
Thus, $a \to_{R(A)} b = \neg a \veebar b$. It is well known that condition (\ref{SP}) holds in every Boolean algebra. In particular, in $R({\bf{A}})$ we have 
\[
\neg a \to_{R(A)} ( \neg b \to_{R(A)} (\neg a \barwedge \neg b)) = 1.
\]

\begin{remark} \label{remark1}
Let ${\bf{A}}$ be a bounded Hilbert algebra. In \cite{BG} it is proved that $a \to_{R(A)} b = a \to b$, for all $a,b \in R(A)$.
\end{remark}

\begin{lemma} \label{lemainfimo}
Let ${\bf{A}} \in \mathrm{SHil}_{0}^{\vee}$. Then $a \barwedge b = a \wedge b$, for all $a,b \in R(A)$.
\end{lemma}
\begin{proof}
Let $a,b \in R(A)$. Then $\neg \neg a = a$ and $\neg \neg b = b$ and by Lemma \ref{lemaSergio} we have 
\[
a \barwedge b = \neg (a \to \neg b) = \neg (\neg a \vee \neg b) = \neg \neg a \wedge \neg \neg b = a \wedge b.
\qedhere
\]
\end{proof}

\begin{theorem} \label{T1}
Let ${\bf{A}} \in \mathrm{SHil}_{0}^{\vee}$. Then (\ref{SP}) is satisfied in ${\bf{A}}$.
\end{theorem}
\begin{proof}
It is a direct consequence of Remark \ref{remark1} and Lemma \ref{lemainfimo}. Indeed, if $a,b \in A$, then $\neg a, \neg b \in R(A)$ and
\[
\neg a \to ( \neg b \to (\neg a \wedge \neg b)) = \neg a \to_{R(A)} ( \neg b \to_{R(A)} (\neg a \barwedge \neg b)) = 1.
\qedhere  
\]
\end{proof}

The following theorem was proved in \cite{BG} for the class of Hilbert algebras with infimum. However, it can be extrapolated to the class of bounded Hilbert algebras with supremum as a special case. We give the proof so that the paper is self-contained.

\begin{lemma} \label{lemaux}
Let $\mathbf{A} \in \mathrm{Hil}_{0}^{\vee}$. Then,
\[
\neg a \to (\neg b \to c) \leq (\neg a \wedge \neg b) \to c,
\] 
for all  $a,b,c\in A.$
\end{lemma}
\begin{proof}
We have that
\[
\begin{array}{llll}
(\neg a \to (\neg b \to c)) \to ((\neg a \wedge \neg b) \to c) = \\
 = (\neg a \wedge \neg b) \to ((\neg a \to (\neg b \to c)) \to c) \\
 = [(\neg a \wedge \neg b) \to (\neg a \to (\neg b \to c))] \to [(\neg a \wedge \neg b) \to c] \\
 = [((\neg a \wedge \neg b) \to \neg a ) \to ((\neg a \wedge \neg b) \to (\neg b \to c)) ] \to [(\neg a \wedge \neg b) \to c]   \\
 = [1 \to ((\neg a \wedge \neg b) \to (\neg b \to c)) ] \to [(\neg a \wedge \neg b) \to c] \\
 = [(\neg a \wedge \neg b) \to (\neg b \to c) ] \to [(\neg a \wedge \neg b) \to c] \\
 = [((\neg a \wedge \neg b) \to \neg b) \to ((\neg a \wedge \neg b) \to c)] \to [(\neg a \wedge \neg b) \to c] \\
 = [1 \to ((\neg a \wedge \neg b) \to c)] \to [(\neg a \wedge \neg b) \to c] \\
 = [(\neg a \wedge \neg b) \to c] \to [(\neg a \wedge \neg b) \to c] \\
 = 1. 
\end{array}
\]
Thus,  $\neg a \to (\neg b \to c) \leq (\neg a \wedge \neg b) \to c$.
\end{proof}

\begin{theorem} \label{T2}
Let ${\bf{A}} \in \mathrm{Hil}_{0}^{\vee}$. Then the following conditions are equivalent:
\begin{enumerate}
\item ${\bf{A}}$ satisfies (1),
\item $(\neg a \wedge \neg b) \to c = \neg a \to (\neg b \to c)$,  for all $a,b,c \in A$.
\end{enumerate}
\end{theorem}

\begin{proof}
$(1) \Rightarrow (2)$ By Lemma \ref{lemaux}, it suffices to show that $(\neg a \wedge \neg b) \to c \leq \neg a \to (\neg b \to c)$.
\[
\begin{array}{llll}  
[(\neg a \wedge \neg b) \to c] \to [\neg a \to (\neg b \to c)]  =  \\ 
     =\neg a \to [((\neg a \wedge \neg b) \to c) \to (\neg b \to c)] \\
     = \neg a \to [\neg b \to (((\neg a \wedge \neg b) \to c) \to c)] \\ 
     = \neg a \to [(\neg b \to ((\neg a \wedge \neg b) \to c)) \to (\neg b \to c)] \\
     = \neg a \to [((\neg b \to (\neg a \wedge \neg b)) \to (\neg b \to c)) \to (\neg b \to c)] \\
     = [(\neg a \to (\neg b \to (\neg a \wedge \neg b))) \to (\neg a \to (\neg b \to c))] \to [\neg a \to (\neg b \to c)] \\
     = [1 \to (\neg a \to (\neg b \to c))] \to [\neg a \to (\neg b \to c)] \\
     = [\neg a \to (\neg b \to c)] \to [\neg a \to (\neg b \to c)] \\
     = 1.
\end{array}
\]
$(2) \Rightarrow (1)$ By assumption, 
\[ 
\neg a \to (\neg b \to (\neg a \wedge \neg b)) = (\neg a \wedge \neg b) \to (\neg a \wedge \neg b) = 1. 
\qedhere
\]
\end{proof}

According to Theorem \ref{T1}, the class of Stonean Hilbert algebras satisfies equation (\ref{SP}). Therefore, we proceed to prove the next proposition, which will be of use subsequently.

\begin{proposition} \label{generaliz lemma sergio}
Let ${\bf{A}} \in \mathrm{SHil}_{0}^{\vee}$ and let $n\in \mathbb{N}$. Then the following conditions are satisfied for every $a_1, \ldots, a_n \in A$:
\begin{enumerate}
\item $\neg \neg (a_1 \vee \ldots \vee a_n) =  \neg \neg a_1 \vee \ldots \vee \neg \neg a_n$,
\item $(a_1, \ldots, a_{n-1} ; \neg a_n) = \neg a_1 \vee \ldots \vee \neg a_{n-1} \vee \neg a_n$.
\end{enumerate}
\end{proposition}
\begin{proof}
$(1)$ We prove it by induction on $n$. We prove for $n=2$. Let $a,b \in A$. By Theorem \ref{T1}, Theorem \ref{T2} and Lemma \ref{lemaSergio},  
\[ 
\begin{array}{lllllll} 
\neg \neg (a \vee b) &=& \neg (\neg a \wedge \neg b) &=& (\neg a \wedge \neg b) \to 0 \\
                                &=&  \neg a \to (\neg b \to 0) &=& \neg a \to \neg \neg b \\
                                &=& \neg \neg a \vee \neg \neg b.              
\end{array}
\]
We assume that the equation is valid for $n$ and we will prove for $n+1$.
\[
\begin{array}{llll}
\neg \neg(a_1 \vee \ldots \vee a_n \vee a_{n+1})  &=&  \neg \neg((a_1 \vee \ldots \vee a_n ) \vee  a_{n+1}) \\
&=&  \neg \neg(a_1 \vee \ldots \vee a_n) \vee  \neg \neg a_{n+1} \\
                                                   &=&  \neg \neg a_1 \vee \ldots \vee \neg \neg a_n  \vee  \neg \neg a_{n+1}. 
\end{array}
\]
$(2)$ We prove it by induction on $n$. By Lemma \ref{lemaSergio} result immediate  for $n=2$. We assume that the equation is valid for $n$ and we will prove for $n+1$. By  Lemma \ref{lemaSergio} and above item we have:
\[
\begin{array}{llll}
(a_1, \ldots, a_n; \neg a_{n+1})  &=&  (a_1, \ldots, a_{n-1}; a_{n} \to \neg a_{n+1}) \\
                                                   &=&  (a_1, \ldots, a_{n-1}; \neg a_{n} \vee\neg a_{n+1}) \\
                                                   &=&  (a_1, \ldots, a_{n-1}; \neg \neg \neg a_{n} \vee \neg \neg \neg a_{n+1}) \\ 
                                                   &=&  (a_1, \ldots, a_{n-1}; \neg\neg (\neg a_{n} \vee \neg a_{n+1}))  \\
                                                   &=&    \neg a_1 \vee \ldots \vee \neg a_{n-1} \vee \neg \neg (\neg a_{n} \vee \neg a_{n+1}) \\
                                                   &=&    \neg a_1 \vee \ldots  \vee \neg a_{n-1} \vee \neg \neg \neg a_{n} \vee \neg \neg \neg a_{n+1} \\
                                                   &=&   \neg a_1 \vee \cdots \vee \neg a_{n-1} \vee \neg a_{n} \vee \neg a_{n+1},
\end{array}
\]
as claimed. 
\end{proof}

With the developed results, we have the following characterization of Stonean Hilbert algebras.

\begin{theorem} \label{charact_stone}
Let ${\bf{A}} \in \mathrm{Hil}_{0}^{\vee}$. Then the following conditions are equivalent:
\begin{enumerate}
\item ${\bf{A}} \in \mathrm{SHil}_{0}^{\vee}$, 
\item $a \veebar b = a \vee b$, for all $a,b \in R(A)$.
\end{enumerate}
\end{theorem}
\begin{proof}
$(1) \Rightarrow (2)$ Let $a,b \in R(A)$. Then $\neg \neg a = a$ and $\neg \neg b = b$. 
Moreover, by Lemma \ref{lemaSergio} and Proposition \ref{generaliz lemma sergio}, we have 
\[
\begin{array}{llllllll}
a \veebar b &=&  \neg \neg (\neg a \to b) &=&  \neg \neg (\neg a \to \neg \neg b) &=&  \neg \neg (\neg \neg a \vee \neg  \neg b)  \\
                   & = & \neg \neg (a \vee b) &=& \neg \neg a \vee \neg \neg b &=& a \vee b.  
\end{array}
\]

$(2) \Rightarrow (1)$ Let $a \in A$. So, $\neg a\in R(A)$ and since $R({\bf{A}})$ is a Boolean algebra, we have $\neg a \veebar \neg\neg a = 1$. By assumption, $\neg a \vee \neg\neg a = 1$. Thus,  ${\bf{A}} \in \mathrm{SHil}_{0}^{\vee}$. 
\end{proof}

\section{$\neg$-morphisms between Stonean Hilbert algebras} \label{sec4}

In this section we introduce and study a particular class of morphisms defined between Hilbert algebras with supremum and Stonean Hilbert algebras, called $\neg$-morphisms.

\begin{definition} 
Let ${\bf{A}}, {\bf{B}} \in \mathrm{Hil}_{0}^{\vee}$. Let $h \colon A \to B$ be a map. We say that $h$ is a $\neg$-morphism if the following conditions are satisfied for every $a,b \in A$:
\begin{enumerate}
\item $h(1) = 1$, 
\item $h(a \to b) \leq h(a) \to h(b)$, 
\item $h(a \vee b) = h(a) \vee h(b)$, 
\item $h(\neg a) = \neg h(a)$.
\end{enumerate}
\end{definition}

In the rest of this paper we will study $\neg$-morphisms between Stonean Hilbert algebras. We denote by ${\mathcal{NSH}}_{0}^{\vee}$ the algebraic category of Stonean Hilbert algebras with $\neg$-morphisms.

\begin{remark} \label{obs morfismos}
Let $h \colon A \to B$ be a $\neg$-morphism between two Stonean Hilbert algebras ${\bf{A}}$ and ${\bf{B}}$. Then $h(0)=0$. Indeed, it follows $h(0) = h(\neg 1) = \neg h(1) = \neg 1 = 0$. Thus, it is clear that every $\neg$-morphism is a $\vee0$-semi-homomorphism. 
\end{remark}

Since every $\neg$-morphism is a $\vee0$-semi-homomorphism, we have the following result.

\begin{lemma} \cite{CM4} \label{ideal y filtro irreduc}
Let ${\bf{A}}, {\bf{B}} \in \mathrm{SHil}_{0}^{\vee}$. Let $h \colon A \to B$ be a $\neg$-morphism. Then:
\begin{enumerate}
\item $( h ( x^{c})]  \in \mathrm{Id}(B)$, for all $x \in \mathrm{X}(A)$,
\item $h^{-1}(y) \in \mathrm{X}(A)$, for all $y \in \mathrm{X}(B)$.
\end{enumerate}
\end{lemma}

\begin{remark}
Not every $\vee0$-semi-homomorphism is a $\neg$-morphism, as we will observe in the following example. Let $A = \{0, a, b, c, 1\}$ with the Hasse diagram of $\langle A, \leq \rangle$ as
follows:
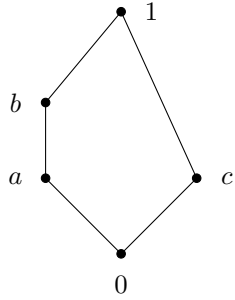
\vspace{0.2cm}
\begin{center}
\begin{tikzpicture}[scale=.8,mipunto/.style ={color=black}]

\filldraw[mipunto] (0,-1.5) circle(2pt);
\filldraw[mipunto] (-1.25,-3) circle(2pt);
\filldraw[mipunto] (-1.25,-4.25) circle(2pt);
\filldraw[mipunto] (1.25,-4.25) circle(2pt);
\filldraw[mipunto] (0,-5.5) circle(2pt);

\filldraw[mipunto] (0,-6) node{$0$};
\filldraw[mipunto] (0.5,-1.5) node{$1$};
\filldraw[mipunto] (-1.75,-3) node{$b$};
\filldraw[mipunto] (-1.75,-4.25) node{$a$};
\filldraw[mipunto] (1.75,-4.25) node{$c$};

\draw (-1.25,-3) -- (0,-1.5);
\draw  (-1.25,-3) -- (-1.25,-4.25);
\draw (-1.25,-4.25) -- (0,-5.5);
\draw (1.25,-4.25) -- (0,-5.5);
\draw (1.25,-4.25) -- (0,-1.5);

\end{tikzpicture}
\end{center}
If we considerer the binary operation $\to$ on $A$ given by 
\vspace{0.2cm}
\begin{center}    
\begin{tabular}{|c|c|c|c|c|c|c|c|c|c|}\hline 
$\to$        & $0$ & $a$ & $b$  & $c$ & $1$ \\ \hline
$0$    & $1$ & $1$ & $1$ & $1$ & $1$ \\ \hline
$a$    & $c$ & $1$ & $1$ & $c$ & $1$ \\ \hline
$b$    & $0$ & $a$ & $1$ & $c$ & $1$ \\ \hline
$c$    & $a$ & $a$ & $b$ & $1$ & $1$ \\ \hline
$1$    & $0$ & $a$ & $b$ & $c$ & $1$ \\ \hline
\end{tabular}
\end{center}
then we have that ${\bf{A}}$ is a bounded Hilbert algebra with supremum. It is easy to prove that ${\bf{A}}$ is a Stonean Hilbert algebra. We consider the map $h \colon A \to A$ given by 
\vspace{0.3cm}
\begin{center}
\begin{tabular}{|c|c|c|c|c|c|c|c|c|c|}\hline 
$x$        & $0$ & $a$ & $b$  & $c$ & $1$ \\ \hline
$h(x)$    & $0$ & $a$ & $a$ & $c$ & $1$ \\ \hline
\end{tabular}
\end{center}
\vspace{0.3cm}
Then $h$ is a $\vee 0$-semi-homomorphism but not a $\neg$-morphism because
\[
h(\neg b)=h(0)=0<c=\neg a=\neg h(b).
\]
\end{remark}

\begin{remark} \label{rem_negmorf}
We also note that a $\neg$-morphism is not necessarily a $\vee0$-homomorphism. Consider the following Stonean Hilbert algebra ${\bf{A}}$ where the binary operation $\to$ is given by order, i.e., $a \to b = 1$, if $a \leq b$ and $a \to b = b$, if $a \nleq b$:
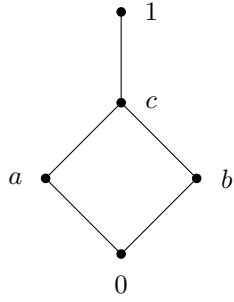
\vspace{0.2cm}
\begin{center}
\begin{tikzpicture}[scale=.8,mipunto/.style ={color=black}]

\filldraw[mipunto] (0,-1.5) circle(2pt);
\filldraw[mipunto] (0,-3) circle(2pt);
\filldraw[mipunto] (-1.25,-4.25) circle(2pt);
\filldraw[mipunto] (1.25,-4.25) circle(2pt);
\filldraw[mipunto] (0,-5.5) circle(2pt);

\filldraw[mipunto] (0,-6) node{$0$};
\filldraw[mipunto] (0.5,-1.5) node{$1$};
\filldraw[mipunto] (0.5,-3) node{$c$};
\filldraw[mipunto] (-1.75,-4.25) node{$a$};
\filldraw[mipunto] (1.75,-4.25) node{$b$};

\draw (0,-1.5) -- (0,-3);

\draw (0,-3) -- (-1.25,-4.25);
\draw (0,-3) -- (1.25,-4.25);
\draw (-1.25,-4.25) -- (0,-5.5);
\draw (1.25,-4.25) -- (0,-5.5);

\end{tikzpicture}
\end{center}
\vspace{0.4cm}
Let $h \colon A \to A$ be the map given by 
\vspace{0.3cm}
\begin{center}
\begin{tabular}{|c|c|c|c|c|c|c|c|c|c|}\hline 
$x$        & $0$ & $a$ & $b$  & $c$ & $1$ \\ \hline
$h(x)$    & $0$ & $c$ & $a$ & $c$ & $1$ \\ \hline
\end{tabular}
\end{center}
\vspace{0.3cm}
Then $h$ is a $\neg$-morphism but not a $\vee0$-homomorphism since 
\[
h(c \to a) = c \neq 1 = h(c) \to h(a).
\]
\end{remark}

The $\neg$-morphisms behave well and transfer the Stone's identity between two Stonean Hilbert algebras.

\begin{theorem}
Let ${\bf{A}}, {\bf{B}} \in \mathrm{Hil}_{0}^{\vee}$. Let $h \colon A \to B$ be a $\neg$-morphism. Then:
\begin{enumerate}
\item If ${\bf{A}}$ is a Stonean Hilbert algebra and $h$ is onto, then ${\bf{B}}$ is a Stonean Hilbert algebra.
\item If ${\bf{B}}$ is a Stonean Hilbert algebra and $h$ is 1-1, then ${\bf{A}}$ is a Stonean Hilbert algebra.
\end{enumerate}
Moreover, if $h$ is bijective, then ${\bf{A}}$ is Stonean Hilbert algebra if and only if ${\bf{B}}$ is a Stonean Hilbert algebra.
\end{theorem}
\begin{proof}
$(1)$ As ${\bf{B}} \in \mathrm{Hil}_{0}^{\vee}$, it remains to prove that $\neg b \vee \neg \neg b=1$, for all $b \in B$. If $b \in B$, then there exists $a \in A$ such that $h(a) = b$. So, since $h$ is a $\neg$-morphism and ${\bf{A}}$ is a Stonean Hilbert algebra, we have  
\[
\neg b \vee \neg \neg b = \neg h(a) \vee \neg \neg h(a) = h(\neg a \vee \neg \neg a) = h(1) = 1.
\]
$(2)$ Let $a \in A$. Then $\neg h(a) \vee \neg \neg h(a) = 1$ because ${\bf{B}}$ is a Stonean Hilbert algebra. Then $h( \neg a \vee \neg \neg a) = h(1)$ and since $h$ is 1-1 it follows $\neg a \vee \neg \neg a = 1$. 
\end{proof}

\begin{theorem} \label{charact_neg_morfismo}
Let ${\bf{A}}, {\bf{B}} \in \mathrm{SHil}_{0}^{\vee}$. Let $h \colon A \to B$ be a $\vee0$-semi-homomorphism. Then the following conditions are equivalent:
\begin{enumerate}
\item $h$ is a $\neg$-morphism,
\item $h^{-1}(y) \in {\mathrm{Max}(A)}$, for all $y \in {\mathrm{Max}(B)}$.
\end{enumerate}
\end{theorem}
\begin{proof}
$(1) \Rightarrow (2)$ Let $y \in {\mathrm{Max}(B)}$. Since ${\mathrm{Max}(B)} \subseteq \mathrm{X}(B)$, $y \in \mathrm{X}(B)$ and  by Lemma \ref{ideal y filtro irreduc}, $h^{-1}(y)$ is an implicative filter. Let $a \in A$ be such that $a \notin h^{-1}(y)$. Then $h(a) \notin y$. As $y$ is a maximal implicative filter, by Lemma \ref{lema2} we have $\neg h(a) \in y$. Since $h$ is $\neg$-morphism, we have $ h(\neg a) \in y$ and so, $\neg a \in h^{-1}(y)$. Therefore, by Lemma \ref{lema2} it follows $h^{-1}(y) \in {\mathrm{Max}(A)}$. 

$(2) \Rightarrow (1)$ Conversely, suppose there exists $a \in A$ such that $\neg h(a) \nleq h(\neg a)$. Then there exists $y \in \mathrm{X}(B)$ such that $\neg h(a) \in y$ and $h(\neg a) \notin y$. By Lemma \ref{lemaSergio} there exists  $z \in {\mathrm{Max}(B)}$ such that $y \subseteq z$. So, $\neg h(a) \in z$ and $h^{-1}(y) \subseteq h^{-1}(z)$. By hypothesis we have $h^{-1}(z) \in {\mathrm{Max}(A)}$ and $h^{-1}(y) \in \mathrm{X}(A)$. Since $h(\neg a) \notin y$, it follows $\neg a \notin h^{-1}(y)$ and by Lemma \ref{lema1} there exists $M \in {\mathrm{Max}(A)}$ such that $h^{-1}(y) \subseteq M$ and $a \in M$. Thus, $M=h^{-1}(z)$ and so, $a \in h^{-1}(z)$, i.e., $h(a) \in z$. Hence, by Lemma \ref{lema1} it follows $\neg h(a) \notin z$, which is a contradiction. 
\end{proof}

\begin{remark}
If ${\bf{A}}$ and ${\bf{B}}$ are two local bounded Hilbert algebras with supremum (see Example \ref{ejemplolocal}), then ${\bf{A}}$ and ${\bf{B}}$ have exactly one maximal implicative filter each, i.e., ${\mathrm{Max}(A)}=\{x\}$ and ${\mathrm{Max}(B)}=\{y\}$. Then, by Theorem \ref{charact_neg_morfismo}, it is immediate to see that a $\vee0$-semi-homomorphism is a $\neg$-morphism if and only if $h^{-1}(y) = x$.
\end{remark}

If ${\bf{A}}$ is a bounded Hilbert algebra, then the structure $\langle R(A), \to, 0, 1 \rangle$ is a bounded Hilbert algebra and the map $r \colon A \to R(A)$ given by $r(x) = \neg \neg x$ is a onto homomorphism of bounded Hilbert algebras (see \cite{Saeid2}). 

\begin{theorem} 
Let ${\bf{A}} \in \mathrm{SHil}_{0}^{\vee}$. Then, $r \colon A \to R(A)$ is a surjective $\vee0$-homomorphism of Stonean Hilbert algebras.
\end{theorem}
\begin{proof}
Let $a,b \in A$. By the results given in \cite{Saeid1} and Theorem \ref{charact_stone} we know that a bounded Hilbert algebra with supremum ${\bf{A}}$ is Stonean if and only if $\langle R(A), \to, \vee, 0, 1 \rangle$ is a Stonean Hilbert algebra. It remains only to prove that $r(a \vee b) = r(a) \veebar r(b)$. Indeed, by Proposition \ref{generaliz lemma sergio} and Theorem \ref{charact_stone}, $ r(a \vee b) = \neg \neg (a \vee b) = \neg \neg a \vee \neg \neg b = r(a) \veebar r(b)$, as claimed. 
\end{proof}

If ${\bf{A}}$ is a bounded Hilbert algebra we know that $R({\bf{A}}) = \langle R(A),  \veebar, \barwedge, \neg, 0, 1 \rangle$ is a Boolean algebra. Now, we show that $\neg$-morphisms between Stonean Hilbert algebras naturally induce Boolean homomorphisms.

\begin{theorem} \label{functor}
Let ${\bf{A}}, {\bf{B}} \in \mathrm{SHil}_{0}^{\vee}$. Let $h \colon A \to B$ be a $\neg$-morphism. Then $\widehat{h} \colon R(A) \to R(B)$ given by $\widehat{h}(x) = \neg \neg h(x)$ is a Boolean homomorphism. Moreover, $r_{\bf{A}} \circ  h = \widehat{h} \circ r_{\bf{B}}$, i.e., the following diagram commute:
\begin{center}
\[
\xymatrix{
\bf{A} \ar[d]_{r_{\bf{A}}} \ar[r]^{h} & \bf{B} \ar[d]^{r_{\bf{B}}}  \\
R({\bf{A}}) ~\ar[r]_{\widehat{h}} & R({\bf{B}}) }
\]
\end{center}
\end{theorem}
\begin{proof}
Let $a,b \in R(A)$. Then it is immediate that $\widehat{h}(1)=1$, $\widehat{h}(0)=0$ and $\widehat{h}(\neg a) = \neg h(a)$. By Lemma \ref{lemaSergio} and since $h$ is a $\neg$-morphism we have 
\[
\begin{array}{lllllllll}
\widehat{h}(a \barwedge b) &=& \neg \neg h(\neg (a \to \neg b)) &=& \neg h(\neg a \vee \neg b) \\ 
                                            &=& \neg (\neg h(a) \vee \neg h(b)) &=& \neg (\neg \neg \neg h(a) \vee \neg \neg \neg h(b)) \\
                                            &=& \neg (\neg \neg h(a) \to \neg \neg \neg h(b)) &=& \widehat{h}(a) \barwedge \widehat{h}(b). 
\end{array}
\]
By Proposition \ref{generaliz lemma sergio} and Theorem \ref{charact_stone}, 
\[
\begin{array}{lllllllll}
\widehat{h}(a \veebar b) &=& \neg \neg h(\neg \neg (\neg a \to b)) &=& \neg \neg h(\neg a \to \neg \neg b) &=& \neg \neg h(a \vee b) \\
                                       &=& \neg \neg (h(a) \vee h(b)) &=& \neg \neg h(a) \vee \neg \neg h(b) &=& \widehat{h}(a) \veebar \widehat{h}(b).                
\end{array}
\]
Hence, $\widehat{h}$ is a Boolean homomorphism between $R({\bf{A}})$ and $R({\bf{B}})$. Finally, since $h$ is a $\neg$-morphism, it follows 
\[
r_{\bf{B}}(h(a)) = \neg \neg h(a) = \neg \neg h(\neg \neg a) = \widehat{h}(r_{\bf{A}}(a)).
\qedhere
\]
\end{proof}

\begin{remark}
By Theorem \ref{functor}, it is clear that the assignments ${\bf{A}} \mapsto R({\bf{A}})$ and $h \mapsto \widehat{h}$ determine a functor from the category of Stonean Hilbert algebras with $\neg$-morphisms and the category of Boolean algebras with Boolean homomorphism.
\end{remark}

\section{Topological duality for ${\mathcal{NSH}}_{0}^{\vee}$} \label{sec5}

In this section we introduce Stonean $H^{\vee}_0$-spaces. We consider the category whose objects are Stonean $H^{\vee}_0$-spaces with certain maps and we prove that it is dually equivalent to the algebraic category whose objects are Stonean Hilbert algebras with $\neg$-morphisms as arrows. 

\begin{definition} \label{Stone_space}
Let $\langle X, \mathcal{T}_{\mathcal{K}} \rangle$ be an $H^{\vee}_0$-space. Then $\langle X, \mathcal{T}_{\mathcal{K}}\rangle$ is a Stonean $H^{\vee}_0$-space if $(U] \in D(X)$, for all $U \in D(X)$.
\end{definition}

\begin{theorem} \label{charact dual space}
Let ${\bf{A}} \in \mathrm{SHil}_{0}^{\vee}$. Then $\langle \mathrm{X}(A), \mathcal{T}_{\mathcal{K}_{A}} \rangle$ is a Stonean $H^{\vee}_0$-space.
\end{theorem}
\begin{proof}
Assume that ${\bf{A}}$ is a Stonean Hilbert algebra. So, $\langle \mathrm{X}(A), \mathcal{T}_{\mathcal{K}_{A}} \rangle$ is an $H^{\vee}_0$-space. It only remains to prove that $(\varphi(a)] \in D(\mathrm{X}(A))$, for all $a \in A$. 

First, we prove that $(\varphi(a)]^{c} \cap \varphi(\neg \neg a) = \emptyset$. In contrast, suppose that there exists $x \in \mathrm{X}(A)$ such that $x \in (\varphi(a)]^{c} \cap \varphi(\neg \neg a)$, i.e., $[x) \cap \varphi(a) = \emptyset$ and $\neg \neg a \in x$. By Lemma \ref{lemaSergio}, there exists a maximal implicative filter $M$ such that $x \subseteq M$ and consequently $\neg \neg a \in M$. Thus, by Lemma \ref{lema2}, $a \in M$ and by Lemma \ref{lema1} we have $\neg a \notin x$. Then, by Lemma \ref{lema1}, there is $y \in \mathrm{X}(A)$ such that $x \subseteq y$ and $a \in y$, or equivalently, $ y \in [x) \cap \varphi(a)$, which is absurd.

From the fact that ${\bf{A}}$ is a Stonean Hilbert algebra and Remark \ref{obs varphi(neg a)}, it follows that
\[
\begin{array}{llllllll}
\mathrm{X}(A) &=& \varphi(1) &=& \varphi(\neg a \vee \neg \neg a) \\
                        &=& \varphi(\neg a) \cup \varphi(\neg \neg a) &=& (\varphi(a)]^{c} \cup \varphi(\neg \neg a).
\end{array}
\]
Thus, $ \varphi(\neg \neg a)=(\varphi(a)]$ and therefore, $(\varphi(a)]\in D(\mathrm{X}(A))$.
\end{proof}

\begin{theorem} \label{charact dual Stone algebra}
Let $\langle X, \mathcal{T}_\mathcal{K} \rangle$ be a Stonean $H^{\vee}_0$-space. Then $\langle D(X), \cup, \Rightarrow, \emptyset, X \rangle$ is a Stonean Hilbert algebra. 
\end{theorem}
\begin{proof}
As  $\langle X, \mathcal{T}_\mathcal{K} \rangle$ is a $H^{\vee}_0$-space then $\langle D(X), \cup, \Rightarrow, \emptyset, X \rangle$ is a bounded Hilbert algebra with supremum. It only remains to prove that $\neg V \cup \neg \neg V = X$, for all $V \in D(X)$. For this, suppose there exists $x \in X$ such that  $x \notin \neg V \cup \neg \neg V$ for some $V \in D(X)$, i.e., $x \notin \neg V = V\Rightarrow \emptyset= (V]^c$ and $x \notin \neg \neg V = ((V]^c]^c.$  Thus, $x \in (V]$ and $x\in ((V]^c]$. So, there exists $y \in (V]^c$ such that $x \leq y.$ By assumption, $(V] \in D(X)$ and as $x \leq y$ we have $y \in(V]$, which is impossible. Therefore, $\neg V \cup \neg \neg V = X$, for all $V \in D(X)$ and $\langle D(X), \cup, \Rightarrow, \emptyset, X \rangle$ is a Stonean Hilbert algebra.
\end{proof}

The fact that not every $\neg$-morphism is a $\vee0$-homomorphism (see Remark \ref{rem_negmorf}) motivates us to find a duality for the category of Stonean Hilbert algebras with $\neg$-morphisms. We consider the spectral maps and impose a condition on them in order to prove that they are the morphisms corresponding to the $\neg$-morphisms defined between Stonean Hilbert algebras.

\begin{definition}\label{def: negmap}
Let $\langle X_{1}, \mathcal{T}_{\mathcal{K}_{1}} \rangle$, $\langle X_{2}, \mathcal{T}_{\mathcal{K}_{2}} \rangle$ be two Stonean $H_{0}^{\vee}$-spaces. Let $f \colon X_{1} \to X_{2}$ be a spectral map. We say that $f$ is a $\neg$-map if satisfies the following condition: 
\begin{equation} \label{SH}
f(x) \leq y \text{ then } \exists z \in X_1 \text{ such that } x \leq z \text{ and } y \leq f(z).
\end{equation}
\end{definition}

Denote by $\mathcal{NMS}_{0}^{\vee}$ the category of Stonean $H_{0}^{\vee}$-spaces with $\neg$-maps.

\begin{theorem}
Let ${\bf{A}}, {\bf{B}} \in \mathrm{SHil}_{0}^{\vee}$. Let $h \colon A \to B$ be a $\neg$-morphism. Then the map $f_{h} \colon \mathrm{X}(B) \to \mathrm{X}(A)$ defined by
\[
f_{h}(x) = h^{-1}(x)
\]
is a $\neg$-map.
\end{theorem}
\begin{proof}
Since $h$ is a  $\neg$-morphism, by Remark \ref{obs morfismos}, $h$ it is a $0\vee$-semi-homomorphism. According to the results proven in \cite{CM4}, $f_{h}$ is a spectral map. It only remains to prove that $f_{h}$ satisfies the condition (\ref{SH}) of Definition \ref{def: negmap}. Assume that $f_h(x) \subseteq y$ for $x \in \mathrm{X}(B)$ and $y \in \mathrm{X}(A)$. Suppose that $0 \in {\mathrm{Fig}}( x \cup h(y))$. So, there exist $q_1, \ldots, q_n \in x$ and $p_1,\ldots, p_m \in y$ such that $(q_1, \ldots, q_n, h(p_1), \ldots, h(p_m); 0) = 1 \in x$. Since $q_1, \ldots, q_n \in x$ and $x \in \mathrm{X}(B)$, $(h(p_1), \ldots, h(p_m); 0 ) \in x$. By Proposition \ref{generaliz lemma sergio}, and since $h$ is a $\neg$-morphism, we have 
\[
\begin{array}{llll}
(h(p_1),\ldots, h(p_m); 0) &=& (h(p_1),\ldots, h(p_{m-1}); h(p_m) \to 0) \\
&=& (h(p_1),\ldots, h(p_{m-1}); \neg h(p_m)) \\
&=& \neg h(p_1) \vee \ldots  \vee\neg h(p_{m-1}) \vee \neg h(p_m) \\
&=& h(\neg p_1 \vee \ldots \vee \neg p_{m-1} \vee  \neg p_m).
\end{array}
\]
It folllows $h(\neg p_1 \vee \ldots \vee  \neg p_m) \in x$. By assumption, $h^{-1}(x) \subseteq y$ and so, $\neg p_1 \vee \ldots \vee \neg p_{m} \in y$. As $y$ is a prime filter, there exists $i \in \{1, \ldots, m \}$ such that $\neg p_{i} \in y$ which is impossible because $p_{i} \in y$. Thus, $0 \notin {\mathrm{Fig}}( x \cup h(y))$ and by Theorem \ref{separacion} there is $z \in \mathrm{X}(B)$ such that  $x \subseteq z$ and $h(y)\subseteq z$, and consequently, $y \subseteq h^{-1}(z) = f_h(z)$. 
\end{proof}

\begin{theorem}
Let $\langle X_{1}, \mathcal{T}_{\mathcal{K}_{1}} \rangle $, $\langle X_{2}, \mathcal{T}_{\mathcal{K}_{2}} \rangle$ be two Stonean $H_{0}^{\vee}$-spaces. Let $f \colon X_{1} \to X_{2}$ be a $\neg$-map. Then, the map $h_{f} \colon D(X_{2}) \to D(X_{1})$ defined by
\[
h_{f}(U) = f^{-1}(U)
\]
is a $\neg$-morphism.
\end{theorem}
\begin{proof} 
Since $f$ is $\neg$-map, by the results given in \cite{CM4}, $h_{f}$ is a $\vee0$-semi-homomorphism. It only remains to prove that $h_{f}(\neg U)=\neg h_{f}(U)$, for all $U\in D(X_2)$. Since $h_{f}$ is a $\vee0$-semi-homomorphism, 
\[
h_{f}(\neg U) = h_{f} (U \Rightarrow \emptyset) \subseteq h_{f}( U) \Rightarrow h_{f}(\emptyset) = \neg h_{f}(U).
\]
Let $x \in \neg h_{f}(U) = \neg f^{-1}(U)$. So, $[x) \cap f^{-1}(U) = \emptyset$. We prove $x \in h_{f}(\neg U) = f^{-1}(\neg U)$  by showing that $f(x) \in U \Rightarrow \emptyset$, that is, $[f(x)) \cap U = \emptyset$. In contrast, suppose that there is $y \in X_2$ such that $f(x) \leq y$ and $y \in U$. So, there exists $z \in X_1$ such that $x \leq z$ and $y \leq f(z)$. As $y \in U$ and $U$ is an upset of $X_2$ we have $f(z) \in U$. Thus, there exists $z \in X_1$ such that $z \in [x)$ and $z \in f^{-1}(U)$, which is impossible. 
\end{proof}

We conclude with the following result.

\begin{theorem}
The categories ${\mathcal{NSH}}_{0}^{\vee}$ and $\mathcal{NMS}_{0}^{\vee}$ are dually equivalent. 
\end{theorem}

\section{$\alpha$-ideals in Stonean Hilbert algebras} \label{sec6}

In \cite{C0} and \cite{Gaitan0} the structure of $\alpha$-deductive systems and $\alpha$-filters in bounded Hilbert algebras with infimum was studied, respectively. In this section, following the results given in \cite{CM1}, we investigate the family of $\alpha$-ideals of a Stonean Hilbert algebra, a notion to be defined in what follows, and their connection with certain subsets of the associated dual space.

\begin{definition}
Let ${\bf{A}} \in \mathrm{Hil}_{0}^{\vee}$. Let $I \in \mathrm{Id}(A)$. We say that $I$ is an $\alpha$-ideal if $\neg \neg a \in I$, whenever $a \in I$.
\end{definition}

Denote by $\mathrm{Id}_{\alpha}(A)$ the set of all $\alpha$-ideals of ${\bf{A}}$. Since $a \leq \neg \neg a$ for all $a \in A$, we have that an ideal $I$ of $A$ is an $\alpha$-ideal if and only if $\forall a \in A$, $a \in I$ iff $\neg \neg a \in I$.

\begin{lemma} \cite{C0} \label{lema_alpha_ideal generated}
Let ${\bf{A}} \in \mathrm{SHil}_{0}^{\vee}$. Let $I \in \mathrm{Id}(A)$. Then 
\[
I^{\alpha} = \{ a \in A \colon \exists x \in I (a \leq \neg \neg x) \}
\]
is the smallest $\alpha$-ideal containing $I$.
\end{lemma}

We define the following operations on $\mathrm{Id}_{\alpha}(A)$:
\begin{eqnarray*}
I \sqcap J                     & := & I \cap J, \\
I \sqcup J                     & := & \{ x \in A \colon \exists i \in I \exists j \in J (x \leq i \vee j) \}, \\
I \twoheadrightarrow J & := & \{ x \in A \colon \forall i \in I \exists j \in J (x \leq i \to j) \}. 
\end{eqnarray*}
Let consider the structure $\mathrm{Id}_{\alpha}({\bf{A}}) = \langle \mathrm{Id}_{\alpha}(A), \sqcup, \sqcap, \twoheadrightarrow, \{0\}, A \rangle$.

\begin{theorem} \cite{C0}
Let ${\bf{A}} \in \mathrm{SHil}_{0}^{\vee}$. Then $\mathrm{Id}_{\alpha}({\bf{A}})$ is a Heyting algebra. 
\end{theorem}

The following concept was introduced in \cite{CM1}. 

\begin{definition}
Let $\langle X, \mathcal{T}_{\mathcal{K}} \rangle$ be an $H_{0}^{\vee}$-space. Let $Y$ be a subset of $X$. We say that $Y$ is an open directed subset if there exists a subset $B \subseteq D(X)$ such that 
\[
Y = \bigcup \{ U \colon U \in B \}.
\]
\end{definition}

Let $\langle X, \mathcal{T}_{\mathcal{K}} \rangle$ be an $H_{0}^{\vee}$-space. Denote by $\mathrm{Od}(X)$ the set of all open directed subsets of $\langle X, \mathcal{T}_{\mathcal{K}} \rangle$. If $Y \subseteq X$, we consider $I(Y) = \{ U \in D(X) \colon U \subseteq Y \}$. Note that $I(Y)$ is an ideal of $D(X)$ and $Y$ is open directed if and only if $Y= \bigcup I(Y)$ (\cite{CM1}).

We can define on $\mathrm{Od}(X)$ the following operations:
\begin{eqnarray*}
Y \barwedge Z            & := & \bigcup I(Y \cap Z), \\
Y \rightarrowtail Z       & := & \bigcup \{ U \in D(X) \colon U \subseteq Y^{c} \cup Z \}. 
\end{eqnarray*}
Note that $Y \rightarrowtail Z = \bigcup I(Y^{c} \cup Z)$ and the structure $\langle \mathrm{Od}(X), \cup, \barwedge, \emptyset, X \rangle$ is a bounded lattice. The following results were developed in \cite{CM1}.

\begin{theorem}
Let $\langle X, \mathcal{T}_{\mathcal{K}} \rangle$ be an $H_{0}^{\vee}$-space. Let $Y, Z \in \mathrm{Od}(X)$. Then:
\begin{enumerate}
\item $Y \subseteq Z$ if and only if $Y \rightarrowtail Z = X$,
\item $Z \subseteq Y \rightarrowtail Z$,
\item If $Y \rightarrowtail Z = X$ and $Z \rightarrowtail Y = X$, then $Y=Z$.
\end{enumerate}
\end{theorem}

\begin{theorem} \label{theo_beta}
Let ${\bf{A}} \in \mathrm{Hil}_{0}^{\vee}$. Then the map $\beta \colon \mathrm{Id}(A) \to \mathrm{Od}(\mathrm{X}(A))$ defined by 
\begin{equation} \label{def_beta}
\beta (I) = \{ x \in \mathrm{X}(A) \colon x \cap I \neq \emptyset \}= \bigcup \{ \varphi(a) \colon a \in I \}
\end{equation}
is a lattice-isomorphism that preserves implication.
\end{theorem}

Consider the following special open directed subsets in a Stonean $H_{0}^{\vee}$-space.

\begin{definition}
Let $\langle X, \mathcal{T}_{\mathcal{K}} \rangle$ be a Stonean $H_{0}^{\vee}$-space and let $Y\subseteq X$. We say that $Y$ is an open $\alpha$-directed subset if there exists a subset $B \subseteq D(X)$ such that 
\[
Y = \bigcup \{ (U] \colon U \in B \}.
\]
\end{definition}

We denote by $\mathrm{Od}_{\alpha}(X)$ the set of all open directed subsets of $\langle X, \mathcal{T}_{\mathcal{K}} \rangle$.

\begin{lemma} \label{lema_directed}
Let $\langle X, \mathcal{T}_{\mathcal{K}} \rangle$ be a Stonean $H_{0}^{\vee}$-space and let $Y\subseteq X$. Then, $Y\in \mathrm{Od}_{\alpha}(X)$  if and only if 
$Y = \bigcup \{ (U] \colon (U] \subseteq Y \text{ and } U \in D(X) \}$.
\end{lemma}
\begin{proof}
Suppose that $Y\in \mathrm{Od}_{\alpha}(X)$. So, there exists $B \subseteq D(X)$ such that $Y = \bigcup \{ (U] \colon U \in B \}$. So, $(U] \subseteq Y$, for all $U \in B$. Note that by Definition \ref{Stone_space} we have $(U] \in D(X)$, for all $U \in D(X)$. Then  
\[
Y = \bigcup \{ (U] \colon U \in B \} \subseteq \bigcup \{ (U] \colon (U] \subseteq Y \text{ and } U \in D(X) \} \subseteq Y, 
\]
i.e., $Y = \bigcup \{ (U] \colon (U] \subseteq Y \text{ and } U \in D(X) \}$. The reciprocal es immediate.
\end{proof}

\begin{theorem}
Let $\langle X, \mathcal{T}_{\mathcal{K}} \rangle$ be a Stonean $H_{0}^{\vee}$-space. Then $\langle \mathrm{Od}_{\alpha}(X), \cup, \barwedge, \emptyset, X \rangle$ is a bounded sublattice of $\langle \mathrm{Od}(X), \cup, \barwedge, \emptyset, X \rangle$. Moreover, $\langle \mathrm{Od}_{\alpha}(X), \cup, \barwedge, \rightarrowtail, \emptyset, X \rangle$ is a Heyting algebra where
\[
Y \rightarrowtail Z = \bigcup \{ (U] \colon (U] \subseteq Y^{c} \cup Z \text{ and } U \in D(X) \},
\]
for all $Y, Z \in \mathrm{Od}_{\alpha}(X)$.
\end{theorem}
\begin{proof}
Let $Y,Z \in \mathrm{Od}_{\alpha}(X)$. By Lemma \ref{lema_directed} it follows $Y \barwedge Z, Y \rightarrowtail Z \in \mathrm{Od}_{\alpha}(X)$. We see that $Y \cup Z \in \mathrm{Od}_{\alpha}(X)$. It is clear that 
\[
\bigcup \{ (U] \colon (U] \subseteq Y \cup Z \text{ and } U \in D(X) \} \subseteq Y \cup Z.
\]
On ther other hand, since $Y,Z \subseteq \bigcup \{ (U] \colon (U] \subseteq Y \cup Z \text{ and } U \in D(X) \}$, we have 
\[
Y \cup Z \subseteq  \bigcup \{ (U] \colon (U] \subseteq Y \cup Z \text{ and } U \in D(X) \}. 
\]
Finally, we prove 
\[
Y \barwedge Z \subseteq W \Longleftrightarrow Y \subseteq Z \rightarrowtail W.
\]
Suppose $Y \barwedge Z \subseteq W$ and $Y \nsubseteq Z \rightarrowtail W$. Then there is $x \in Y$ such that $x \notin Z \rightarrowtail W$. So, there exists $U_{0} \in D(X)$ such that $(U_{0}] \subseteq Y$, $x \in (U_{0}]$ and $(U_{0}] \nsubseteq Z^{c} \cup W$. Thus, there exists $y \in (U_{0}] \subseteq Y$ and $y \in Z \cap W^{c}$. Consequently, $y \in Y \cap Z$ and by assumption, $y \in W$ which is a contradiction. 

Conversely, suppose $Y \subseteq Z \rightarrowtail W$. Let $x \in Y \barwedge Z$. So, $x \in Z \rightarrowtail W$, i.e., there exists $U_{0} \in D(X)$ such that $x \in (U_{0}]$ and $(U_{0}] \subseteq Z^{c} \cup W$. As $x \in Z$ and $x \in Z^{c} \cup W$ we have $x \in W$. Therefore, $Y \barwedge Z \subseteq W$.
\end{proof}

Let ${\bf{A}} \in \mathrm{SHil}_{0}^{\vee}$. By restricting the homomorphism $\beta$ of the Theorem \ref{theo_beta} to the set of $\alpha$-ideals we obtain that the Heyting algebra of $\alpha$-ideals of {\bf{A}} is isomorphic to the Heyting algebra of open $\alpha$-directed of the dual space $\langle \mathrm{X}(A), \mathcal{T}_{\mathcal{K}_{A}} \rangle$.

\begin{theorem}
Let ${\bf{A}} \in \mathrm{SHil}_{0}^{\vee}$. Then the map $\beta \colon \mathrm{Id}_{\alpha}(A) \to \mathrm{Od}_{\alpha}(\mathrm{X}(A))$ defined by (\ref{def_beta}) is a Heyting-isomorphism. Moreover, 
\[
\beta(I) = \bigcup \{ (\varphi(a)] \colon \neg \neg a \in I \},
\]
for all $I \in \mathrm{Id}_{\alpha}(A)$.
\end{theorem}
\begin{proof}
First we prove $\beta(I) = \bigcup \{ (\varphi(a)] \colon \neg \neg a \in I \}$, for all $I \in \mathrm{Id}_{\alpha}(A)$.  Let $I \in \mathrm{Id}_{\alpha}(A)$ and $x \in \beta(I)$. Then, by Theorem \ref{theo_beta}, there exists $a \in I$ such that $x \in \varphi(a)$, i.e., $a \in x$. As $a \leq \neg \neg a$ we have $\neg \neg a \in x$ and so, $x \in \varphi(\neg \neg a)$. By Theorem \ref{charact dual space} it follows  $x\in (\varphi(a)]$. Moreover, since $I$ is an $\alpha$-ideal and $a\in I$, $\neg \neg a \in I$ . Consequently, $x \in \bigcup \{ (\varphi(a)] \colon \neg \neg a \in I \}$. Conversely, if $x \in \bigcup \{ (\varphi(a)] \colon \neg \neg a \in I \}$, then there exists $\neg \neg a \in I$ such that $x \in (\varphi(a)]=\varphi(\neg \neg a)$. So, $\neg \neg a \in x \cap I$ and $x \in \beta(I)$.

By Theorem \ref{theo_beta}, the map $\beta$ is a lattice-isomorphism such that $\beta(I \twoheadrightarrow J) = \beta(I) \rightarrowtail \beta(J)$, for all $I,J \in \mathrm{Id}_{\alpha}(A)$. Hence, $\beta$ is a Heyting-homomorphism.

It only remains to prove that $\beta$ is onto. Let $Y \in \mathrm{Od}_{\alpha}(\mathrm{X}(A))$. Then there exists $B \subseteq A$ such that $Y = \bigcup \{ (\varphi(a)] \colon a \in B \}$. Consider the ideal generated $\mathrm{Idg}(B)$. By Lemma \ref{lema_alpha_ideal generated}, the set
\[
\mathrm{Idg}(B)^{\alpha} = \{ a \in A \colon \exists x \in \mathrm{Idg}(B) (a \leq \neg \neg x) \}
\]
is an $\alpha$-ideal. We prove that $\beta(\mathrm{Idg}(B)^{\alpha})=Y$. If $x \in \beta(\mathrm{Idg}(B)^{\alpha}) = \bigcup \{ (\varphi(a)] \colon \neg \neg a \in \mathrm{Idg}(B)^{\alpha} \}$, then there exists $\neg \neg a \in \mathrm{Idg}(B)^{\alpha}$ such that $x \in (\varphi(a)] = \varphi(\neg \neg a)$. So, there is $b \in \mathrm{Idg}(B)$ such that $\neg \neg a \leq \neg \neg b$, i.e., there exist $b_{1}, \ldots, b_{n} \in B$ such that $b \leq b_{1} \vee \ldots \vee b_{n}$ and $\neg \neg a \leq \neg \neg b$. By Proposition \ref{generaliz lemma sergio}, 
\[
\neg \neg a \leq \neg \neg b \leq \neg \neg (b_{1} \vee \ldots \vee b_{n}) = \neg \neg b_{1} \vee \ldots \vee \neg \neg b_{n}.
\]
Since $\neg \neg a \in x$ we have $\neg \neg b_{1} \vee \ldots \vee \neg \neg b_{n} \in x$ and as $x$ is a prime implicative filter, there is $i \in \{1, \ldots, n \}$ such that $\neg \neg b_{i} \in x$. So, $x \in \varphi(\neg \neg b_{i}) = (\varphi(b_{i})]$, with $b_i \in B$. Consequently,  $x \in \bigcup \{ (\varphi(a)] \colon a \in B \} = Y$.

For the other inclusion, let us take $x \in Y$. Then there is $a \in B$ such that $x \in (\varphi(a)] = \varphi(\neg \neg a)$. Thus, $\neg \neg a \in x$. Since $a \in B$, we have $a \in \mathrm{Idg}(B)$ and $\neg \neg a \in \mathrm{Idg}(B)^{\alpha}$. Hence, $\neg \neg a \in x \cap \mathrm{Idg}(B)^{\alpha}$ and $x \in \beta(\mathrm{Idg}(B)^{\alpha})$. Therefore, $\beta(\mathrm{Idg}(B)^{\alpha}) = Y$.
\end{proof}

\end{document}